\documentclass[draft]{amsart}

 \usepackage{amssymb}

\newtheorem{thm}{Theorem}[section]

\newtheorem{lm}[thm]{Lemma}

\theoremstyle{remark}
\newtheorem{re}[thm]{Remark}

\theoremstyle{definition}

\numberwithin{equation}{section}
\numberwithin{thm}{section}

\begin{document}



\subjclass{Primary 32-XX; Secondary 46-XX}



\title{The cluster value problem for Banach spaces}


%

\author{W.~B.~Johnson}
\address{Department of Mathematics, Texas A\&M University, College Station, TX 77843}
\curraddr{}
\email{johnson@math.tamu.edu}

\author{S.~Ortega Castillo}
\address{Department of Mathematics, Texas A\&M University, College Station, TX 77843}
\curraddr{CIMAT A.~C., Guanajuato, Guanajuato, M\'exico 36240}
\email{sofia.ortega@cimat.mx}


%

\thanks{The authors were supported in part by NSF DMS 10-01321.}

%

\begin{abstract}
The main result is that the cluster value problem in separable Banach spaces, for the Banach algebras $A_u$ and $H^{\infty}$, can be reduced to the cluster value problem in those spaces which are $\ell_1$ sums of a sequence  of finite dimensional spaces. In particular we prove that the cluster value problem for $\ell_1$ is equivalent to the cluster value problem for $L_1(0,1)$.
\end{abstract}

\maketitle



\section{Introduction}

The cluster value problem was introduced in 1961 in \cite{IJS} by a group of mathematicians who worked together under the alias I.~J.~Schark. They compared the  set of cluster values of a function $f$  in an algebra $H$ of bounded analytic functions on a domain $U$ at a point $z\in \overline{U}$ with the set $\{ \gamma(f) : \gamma \in \mathcal{M}_z \}$ where $\mathcal{M}_z$ is the set of multiplicative linear functionals on the algebra $H$ such that $\gamma(L) = L(z)$ for all linear functionals.  They  found these sets are the same when the domain $U$ is the unit disk in the complex plain. It was soon understood that this cluster value theorem would be a consequence of a positive solution to the  famous corona problem, which of course was solved positively in the unit disk of the complex plane by Carleson in 1962. In 1979 the cluster value problem was solved positively by McDonald in \cite{Mc} for bounded analytic functions defined on a strongly pseudoconvex domain with smooth boundary in $\mathbb{C}^n$.  But only in 2012 did Aron, Carando, Gamelin, Lasalle and Maestre consider the case of domains $U$ in  infinite dimensional Banach spaces. They  solved the cluster value problem at the origin for bounded, uniformly continuous and analytic functions defined on the ball of a Banach space with a shrinking $1$-unconditional basis in \cite{ACGLM}. This class of spaces  include, for example,  the $\ell_p$ spaces for $1<p<\infty$ as well as $c_0$, but not, for example, the space $\ell_1$ or  $\ell_\infty$ or $L_p(0,1)$, $1\le p \not= 2 \le \infty$.  In \cite{JO}  we gave a positive answer to the cluster value problem for bounded analytic functions on the ball of the space of continuous functions on a dispersed compact Hausdorff  space.

\medskip  

This paper focuses  attention on  the cluster value problem for the unit ball of the $\ell_1$ sum of finite dimensional spaces. We prove that in order to solve the cluster value problem for the unit  ball of all separable Banach spaces, it is enough to solve it for the ball of spaces that are an $\ell_1$ sum of a sequence of  finite dimensional spaces. That is, if the cluster value problem has a negative solution for some separable Banach space, then it has a negative answer for some space that is an $\ell_1$ sum of finite dimensional spaces. Although we do not see a reduction of the cluster value problem for all Banach spaces to the case of separable Banach spaces, an obvious modification of the proof of Theorem \ref{maintheorem} yields that if the cluster value problem has a positive solution for (uncountable)  $\ell_1$ sums of all collections of finite dimensional spaces, then there is a positive answer for all Banach spaces.

\medskip

To be precise, we recall \cite{JO} that the cluster value problem is posed as follows: Given a Banach space $X$, an algebra $H$ of bounded analytic functions over the ball $B$ of $X$ that contains $X^*$, and a point $x^{**}$ in the weak$^*$ closure of $B_{X^{**}}$ (that is, in the closed unit ball of $X^{**}$), we define, for $f \in H$,  the set $Cl_B(f, x^{**})$ to be all the limit values of $f(x_\alpha)$ over all nets $(x_\alpha)$ in $B_{X}$ converging to $x^{**}$ in the weak$^*$ topology, and we define $M_{x^{**}}$ to be the multiplicative linear functionals in the spectrum over $H$ that, when restricted to $X^*$, coincide with $x^{**}$. Is it then true that $Cl_B(f, x^{**}) \supset M_{x^{**}}(f)$ for every $f \in H$? The reverse inclusion is always true as we may recall from the first section in \cite{JO}.

\medskip

The algebras of analytic functions we consider in this paper are $H^{\infty}(B)$, the set of all bounded analytic functions on the open unit  ball $B=B_X$ of the Banach space $X$, and $A_u(B)$, the uniformly continuous functions that are in $H^{\infty}(B)$.

\section{Reduction of the cluster value problem to $\ell_1$ sums of finite dimensional spaces}

We will see that the cluster value problem in Banach spaces can be reduced to the cluster value problem in those spaces that are $\ell_1$ sums of finite dimensional spaces. For simplicity, we will only show that the cluster value problem for separable spaces can be reduced to the cluster value problem in spaces that are a countable $\ell_1$ sum of finite dimensional spaces, and it will be clear that the nonseparable case reduces to uncountable $\ell_1$ sums of finite dimensional spaces.

\medskip

We will need the next two lemmas. The ideas in the following lemmas originated with C. Stegall \cite{S} and were developed by the first author in \cite{J}.

\bigskip

\begin{lm} \label{lemma1} Let $Y$ be a separable Banach space  and $Y_1\subset Y_2\subset Y_3 \subset \dots$ an increasing sequence of finite dimensional subspaces whose union is dense in $Y$. Set  $X=(\sum Y_n)_1$. Then the isometric quotient map $Q: X \to Y$ defined by $$Q(z_n)_n :=\sum_{n=1}^{\infty} z_n$$ induces an isometric algebra homomorphism $Q^\#:H(B_Y) \to H(B_X),$ where $H$ denotes either the algebra $A_u$ or the algebra $H^{\infty}.$
\end{lm}

\begin{proof}

Note that for all $(z_n)_n \in X$, 
\begin{equation}\label{eq1}
\|Q(z_n)_n\|=\|\sum_{n=1}^{\infty} z_n\| \leq \sum_{n=1}^{\infty}\|z_n\|=\|(z_n)_n\|_1.
\end{equation}

Let $\widetilde{Y_n}=\{ (z_n)_n \in X: \; z_k=0 \; \forall \; k \neq n\}.$ Since $Q(B_{\widetilde{Y_n}})=B_{Y_n}$ for all $n \in \mathbb{N},$ we now have that $Q(B_X)$ is dense in $ B_Y$ and hence $Q$ is an isometric quotient map.

\medskip

Then the function $Q^\#:H(B_Y) \to H(B_X)$ given by $Q^\#(f)=f\circ Q$ is an isometric homomorphism because $Q^\#$ is clearly linear and for all $f \in H(B_Y)$,  
$$
\|Q^\#(f)\|=\sup_{x\in B_X}|f\circ Q(x)|=\sup_{y \in B_Y}|f(y)|=\| f \|.
$$
Moreover, for all $f, g \in H(B_Y)$,  
$$
Q^\#(f\cdot g)=(f\cdot g) \circ Q=(f\circ Q) \cdot (g\circ Q)=Q^\#(f)Q^\#(g),
$$
so $Q^\#$ is an algebra homomorphism.
\end{proof}

\begin{lm}\label{lemma2}
Under the assumptions of Lemma \ref{lemma1}, there is a norm one algebra homomorphism $T: H(B_X) \to H(B_Y)$ so that $T(X^*)\subset Y^*$ and $T \circ Q^\#=I_{H(B_Y)}$.
 \end{lm}

\begin{proof}

The first part of the proof consists of constructing $T$ and verifying that $TH(B_X) \subset H(B_Y)$. 

\medskip

For every $y \in (\cup Y_n)$ and $n \in \mathbb{N},$ let $S_n(y)=(z_i)_i \in X$ be given by

$$z_i =\left\lbrace \begin{array}{ll}
         y & \text{  if } i=n \text{ and } y \in Y_n\\
         0 & \text{ otherwise} 
         \end{array} \right. 
         $$

Let $\mathcal{U}$ be a free ultrafilter on $\mathbb{N}$.   For each $g \in H(B_X)$ set  $$Sg(y)=\lim_{n \in \mathcal{U}} g(S_n y) \text{ for every } y \in B_{(\cup Y_n)},$$ which is well defined because $g$ is bounded. Next we prove that $Sg$ is continuous:

\medskip

Let $0<r<1.$ Since $g \in H^{\infty}(B_X)$ then Schwarz' Lemma (Thm. 7.19, \cite{M}) and the convexity of $B_X$ imply that $g \in A_u(r B_X).$ Let $\epsilon >0.$ Since $g$ is uniformly continuous on $rB_X$ there exists $\delta>0$ such that, if $a,b \in r B_X$ and $\|a-b
\|<\delta,$ then $\|g(a)-g(b)\|<\epsilon.$ Thus, given $y_1, y_2 \in r B_{(\cup Y_n)}$ such that $\|y_1-y_2\|<\delta$, we can find $N \in \mathbb{N}$ such that $y_1,y_2 \in Y_n \; \; \forall \; n\geq N,$ and then $\|S_n(y_1)-S_n(y_2)\|=\|y_1-y_2\|<\delta$ eventually for $n$, so
$$\|Sg(y_1)-Sg(y_2)\|=\lim_{n \in \mathcal{U}}\|g(S_n(y_1))-g(S_n(y_2))\| \leq \epsilon.$$

Since each $Sg: B_{(\cup Y_n)} \to \mathbb{C}$ is uniformly continuous on $rB_{(\cup Y_n)}$ for $0<r<1,$ we can continuously extend each $Sg$ to $Tg: B_Y \to \mathbb{C}.$ Moreover, it is evident that $Tg$ is uniformly continuous on $B_X$ when $g \in A_u(B_X).$ It is left to show that each $Tg$ is analytic by checking that every $Tg$ is analytic in each complex line (Thm. 8.7, \cite{M}). We do this in two parts.

\medskip

\textbf{Step 1} Let us check that each $Sg$ is analytic on complex lines:

\medskip

Let $y^1 \in B_{(\cup Y_n)}$ and $y^2 \neq 0 \in {(\cup Y_n)}.$ Since $B_{(\cup Y_n)}$ is open we can find $R>0$ such that, if $\|y-y^1\|<R$, then $y \in B_{(\cup Y_n)}.$ Choose $r>0$ such that $r \|y^2\|< R.$ Thus, if $|\lambda|\leq r$ we have that $\lambda \in \Lambda=\{ \zeta \in \mathbb{C}: \; y^1+\zeta y^2 \in B_{(\cup Y_n)}\}.$

\medskip 

For each $n \in \mathbb{N},$ let $u_n=S_n(y^1)$ and $w_n=S_n(y^2).$ 

\medskip

Using the notation in Remark 5.2 in \cite{M}, we claim that
$$Sg(y^1+\lambda y^2)=\sum_{m=0}^{\infty} ( \lim_{n \in \mathcal{U}} P^m g (u_n)(w_n))\lambda^m,$$
uniformly on $\lambda,$  for $|\lambda|\leq r.$ Let us start by showing that for each $m \in \mathbb{N}$ and $\lambda \in \bar{\Delta}(0,r),$ $\lim_{n \in \mathcal{U}} P^m g (u_n)( \lambda w_n)$ exists.

\medskip

We can find $s>1$ such that also $sr\|y^2\|<R.$ Then, when $|t| \leq s$ and $|\lambda|\leq r$, we have that $y^1+t\lambda y^2 \in B_{(\cup Y_n)},$ because

$$\|(y^1+t\lambda y^2)-y^1\| = |t\|\lambda| \|y^2\| \leq s r \|y^2\| < R,$$

so $u_n+t\lambda w_n \in B_X$ eventually for n, and from Cauchy's Inequality (Cor. 7.4, \cite{M}), for each $m \in \mathbb{N}$
$$\|P^m g(u_n)(\lambda w_n)\|\leq \frac{1}{s^m}\|g\|,$$
eventually for n, and then $\lim_{n \in \mathcal{U}} P^m g (u_n)( \lambda w_n)$ exists.

\medskip

Moreover, given $M \in \mathbb{N}$ and $\lambda$ such that $|\lambda|\leq r,$
\begin{align*}
\| Sg(y^1+\lambda y^2)&-\sum_{m=0}^M \lim_{n \in \mathcal{U}} P^m g (u_n)(w_n))\lambda^m\|\\
&=\|\lim_{n \in \mathcal{U}} (g(u_n+\lambda w_n)-\sum_{m=0}^M P^m g (u_n)(\lambda w_n) )\|\\
&=\|\lim_{n \in \mathcal{U}} \sum_{m=M+1}^{\infty} P^m g (u_n)(\lambda w_n)\|\\
&\leq \sum_{m=M+1}^{\infty} \tfrac{1}{s^m} \|g\|=\tfrac{\|g\|}{s-1} \tfrac{1}{s^M},
\end{align*}
which goes to zero as $M \to \infty.$

\medskip

Thus $Sg$ is analytic on complex lines.

\bigskip

\textbf{Step 2} The following general lemma should be known, but we could not find a reference.

\medskip

\begin{lm} \label{lemma3}  If $\phi:B_Y \to \mathbb{C}$ is bounded and uniformly continuous on $sB_Y$ for each $0<s<1,$ and there is a dense subspace $Z$ of $Y$ such that $\phi|_{B_Z}$ is analytic on complex lines, then $\phi$ is analytic on complex lines in $B_Y$ (and hence analytic).
\end{lm}
\begin{proof}

Let $y^1 \in B_{Y}$ and $y^2 \neq 0 \in {Y}.$ Let $s \in (\|y^1\|,1)$. Since $sB_{Y}$ is open  and contains $y^1$, we can find $R>0$ such that, if $\|y-y^1\|<R$ then $y \in sB_{Y}.$ Choose $r>0$ such that $r \|y^2\|\leq R.$ Thus, if $|\lambda|< r$ we have that $\lambda \in \Lambda=\{ \zeta \in \mathbb{C}: \; y^1+\zeta y^2 \in sB_{Y}\}.$

\medskip

Let $f: \lambda \to \phi(y^1+\lambda y^2)$, a function defined for $|\lambda|< r$. We want to show that $f$ is analytic. 

\medskip

Let $\{ y_k^1\}_k \subset B_{Z}$ and $\{ y_k^2\}_k \subset Z$ be sequences such that $\|y_k^1-y^1\|\leq \frac{1}{2^k}$ and $\|y_k^2-y^2\|\leq \frac{1}{2^k}.$ Choose $K_1 \in \mathbb{N}$ such that $\frac{1+r}{2^{K_1}}\leq R-r \|y^2\|.$ Then for $k \geq K_1$ and $|\lambda|< r$ we have that 
\begin{align*}
\|(y_k^1+\lambda y_k^2)-y^1\|&\leq \|y_k^1-y^1\|+|\lambda\||y_k^2-y^2\| +|\lambda| \|y^2\| \\ 
&< \tfrac{1+r}{2^{K_1}}+r \|y^2\| \\
&\leq R,
\end{align*}
so $y_k^1+\lambda y_k^2 \in sB_Z$.

\medskip

For each $k\geq K_1,$ let $f_k: \lambda \to \phi(y_k^1+\lambda y_k^2),$ which is an analytic function for $|\lambda|< r$ by assumption.

\medskip

Since $\phi$ is bounded, clearly $\{f_k\}_{k\geq K_1}$ is uniformly bounded. Let us now show that  $\{f_k\}_{k\geq K_1}$ converges uniformly to $f$. Let $\epsilon>0$. Since $\phi$ is uniformly continuous on $sB_Z$, we can find $\delta>0$ such that,
$$a,b \in sB_Z, \; \; \; ||a-b||<\delta  \Longrightarrow ||\phi(a)-\phi(b)||<\epsilon.$$

Choose $K\geq K_1$ such that $\frac{1+r}{2^K}<\delta$. Then $\forall k\geq K$ and $\lambda$ with $|\lambda|<r$ ,

\begin{align*}
\|(y_k^1+\lambda y_k^2)-(y^1+\lambda y^2)\|&\leq \|y_k^1-y^1\|+|\lambda| \|y_k^2-y^2\|\\
&\leq \frac{1+r}{2^k}\\
&<\delta,
\end{align*}

so $\|f_k(\lambda)-f(\lambda)\|=\|\phi(y_k^1+\lambda y_k^2)-\phi(y^1+\lambda y^2)\|<\epsilon$.

\medskip

Then, by the lemma on p.~226 in \cite{A}, $f$ is analytic.

\end{proof}

From the previous two steps, we obtain that $T$ is a well defined mapping from $H(B_X)$ into $H(B_Y).$ Now, given $x^* \in X^*,$ $y^1, y^2 \in B_Y$ and $\lambda \in \mathbb{C}$ such that $y^1+\lambda y^2 \in B_Y,$ we can find $\{ y_k^1\}_k \subset B_{(\cup Y_n)}$ converging to $y^1$ and $\{ y_k^2\}_k \subset B_{(\cup Y_n)}$ converging to $y^2,$ and then
\begin{align*}
Tx^*(y^1+\lambda y^2)&=\lim_{k\to \infty}\lim_{n \in \mathcal{U}} x^*(S_n(y_k^1+\lambda y_k^2))\\
&=\lim_{k\to \infty}\lim_{n \in \mathcal{U}} (x^*(S_n(y_k^1))+\lambda x^*(S_n(y_k^2)))\\
&=Tx^*(y^1)+\lambda Tx^*(y^2),
\end{align*}
i.e. $Tx^* \in Y^*.$ This shows that $TB_{X^*}=B_{Y^*}.$

\medskip

Moreover, for every $f \in H(B_Y)$ and $y \in B_Y,$ we can find $\{ y_k\}_k \subset B_{(\cup Y_n)}$ converging to $y,$ and thus
 \begin{align*}
 T\circ Q^\#(f)(y)&=T(f \circ Q)(y)\\
 &=\lim_{k\to \infty}\lim_{n \in \mathcal{U}} f \circ Q(S_n(y_k))\\
 &=\lim_{k\to \infty} f(y_k)\\
 &=f(y),
 \end{align*}
 so $T \circ Q^\#=I_{H(B_Y)}.$

\medskip

Also, $T$ is a homomorphism because $T$ is clearly linear and for all $f, g \in H(B_X),$ $y \in B_Y,$ we can find $\{ y_k\}_k \subset B_{(\cup Y_n)}$ converging to $y,$ so
\begin{align*}
T(f \cdot g)(y)&=\lim_{k\to \infty}\lim_{n \in \mathcal{U}} f \cdot g(S_n(y_k))\\
&=\lim_{k\to \infty}\lim_{n \in \mathcal{U}} f(S_n(y_k)) \cdot g(S_n(y_k))\\
&=Tf(y)\cdot Tg(y)\\
&=(Tf\cdot Tg)(y).
\end{align*}

Finally, for every $f \in H(B_X),$
$$\|Tf\|=\sup_{y \in B_Y}|Tf(y)|=\sup_{y \in B_{(\cup Y_n)}}|\lim_{n \in \mathcal{U}} f(S_n(y))|\leq \sup_{x \in B_X}|f(x)|=\|f\|,$$
and $\|T\|=\|T\| \|Q^\#\|\geq \|T\circ Q^\#\|=1.$ So $\|T\|=1.$
\medskip
 
\end{proof}

\begin{thm} \label{maintheorem}  Let $Y$ be a separable Banach space  and $Y_1\subset Y_2\subset Y_3 \subset \dots$ an increasing sequence of finite dimensional subspaces whose union is dense in $Y$. Set  $X=(\sum Y_n)_1$. Let $H$ denote either the algebra $A_u$ or the algebra $H^{\infty}.$ If $H(B_X)$ satisfies the cluster value theorem at every $x^{**} \in \overline{B_X}^{**}$ then $H(B_Y)$ satisfies the cluster value theorem at every $y^{**} \in \overline{B_Y}^{**}$.
\end{thm}

\begin{proof}
We know that $\hat{f}(M_{x^{**}}(B_X)) \subset Cl_{B_X}(f, x^{**}),$ for all $f \in H(B_X)$ and $x^{**} \in \overline{B_X}^{**}.$ Let us show that $\hat{g}(M_{y^{**}}(B_Y)) \subset Cl_{B_Y}(g, y^{**})$ for all $g \in H(B_Y)$ and $y^{**} \in \overline{B_Y}^{**}.$

\medskip

Let $y^{**} \in \overline{B_Y}^{**},$ $\tau \in M_{y^{**}}(B_Y)$ and $g \in H(B_Y).$ Let $T$ be the algebra homomorphism from $H(B_X)$ to $H(B_Y)$ constructed in Lemma \ref{lemma2}. Then $x^{**}=y^{**}\circ T \in \overline{B_X}^{**},$ $Q^{\#}(g) \in H(B_X)$ and defining $\tilde{\tau}=\tau \circ T,$ we see that $\tilde{\tau} \in M_{x^{**}}(B_X)$ because for all $x^* \in X^*,$
$$\tilde{\tau}(x^*)=\tau (T x^*)=<y^{**}, T x^*>=<x^{**}, x^*>.$$

Moreover,

$$\widehat{Q^{\#}(g)}(\tilde{\tau})=\tilde{\tau}(Q^{\#}(g))=\tau(T \circ Q^{\#}(g))=\tau(g)=\hat{g}(\tau)$$
and
$$Cl_{B_X}(Q^{\#}(g), x^{**})=Cl_{B_X}(g\circ Q, x^{**}) \subset Cl_{B_Y}(g,Q^{**}x^{**})=Cl_{B_Y}(g, y^{**}),$$
so the theorem is established.
\end{proof}

\begin{re} \label{remark1}  \rm{A very special case of Theorem \ref{maintheorem} is that if $\ell_1$ satisfies  the cluster value theorem, then so does $L_1$.  We do not know for $1<p \not= 2 < \infty$ whether the cluster value theorem for $\ell_p$ implies the cluster value theorem for $L_p$.  Incidentally, in \cite{ACGLM} it was proved that $\ell_p$ for $p$ in this range satisfies the cluster value theorem at $0$, but it is open whether $L_p$ satisfies the cluster value theorem at any point of $B_{L_p}$. }
\end{re}

\begin{re} \label{remark2} \rm{The analogue of Theorem \ref{maintheorem} for non separable spaces, which is proved by a non essential modification of the proof of  Theorem \ref{maintheorem},  can be stated as follows.}  \textit{Let $Y$ be a Banach space and $(Y_\alpha)_{\alpha \in A}$ a family of finite dimensional subspaces of $Y$ that is directed by inclusion and whose union is dense in $Y$.  If $(\sum_{\alpha \in A} Y_\alpha)_1$ satisfies the cluster value theorem, then so does $Y$.}
\end{re}

\begin{re} \label{remark3} \rm{There is a slight strengthening of Theorem \ref{maintheorem}.}  \textit{Let $Y$, $(Y_n)_n$, and $H$ be as in the statement of Theorem  \ref{maintheorem} and suppose that $(X_n)_n$ is a sequence so that $X_n$ is $1+\epsilon_n$-isomorphic to $Y_n$ and $\epsilon_n \to 0$.  If $(\sum_n X_n)_1$ satisfies the cluster value theorem for the algebra $H$, then so does $Y$.} \rm{Now let $(Z_n)$ be a sequence of finite dimensional spaces so that for every finite dimensional space $Z$ and every $\epsilon >0$, the space $Z$ is $1+\epsilon$-isomorphic to one (and hence infinitely many) of the spaces  $Z_n$. Set $C_1 = (\sum_n Z_n)_1$.} As an immediate consequence of this slight improvement of Theorem \ref{maintheorem}  we get  \textit{If $C_1$ satisfies the cluster value theorem for $H$, then so does every separable Banach space.}
\end{re}

The proof of the improved Theorem \ref{maintheorem} is essentially the same as the proof of the theorem itself. One just needs to define in Lemma \ref{lemma1} the   mapping $Q$ so that the conclusion of  Lemma \ref{lemma1} remains true: For each $n$ take an isomorphism $J_n: X_n \to Y_n$ so that for $x\in X_n$ the inequality $(1+\epsilon_n)^{-1} \|x\| \le \|J_n x\| \le \|x\|$ is valid,  and define $Q( x_n)_n   = \sum_n J_n x_n$   for $(x_n)_n$ in  $(\sum_n X_n)_1$.

\begin{re} \label{remark4} \rm{If every Banach space with an unconditional basis satisfies the cluster value theorem for the algebra $H$, then so does every separable Banach lattice. The proof is basically the same as the previous because a Banach lattice is paved by finite dimensional subspaces $(E_n)_n$ with $E_n \subset E_{n+1}$ and $E_n$ has a
$1 + 1/N$ unconditional basis.}
\end{re}


%


\begin{thebibliography}{}

\bibitem{A} Lars V.~Ahlfors, \emph{Complex Analysis, An Introduction to the Theory of Analytic Functions of One Complex Variable}, International Series in Pure and Applied Mathematics, Third Edition, 1979.

\bibitem{ACGLM} R.~M.~Aron, D.~Carando, T.~W.~Gamelin, S.~Lasalle, M.~Maestre, \emph{Cluster Values of Analytic Functions on a Banach space}, Math. Ann.~353 (2012), pp.~293-303.

\bibitem{IJS} I.~J.~Schark, \emph{Maximal Ideals in an Algebra of Bounded Analytic Functions}, Journal of Mathematics and Mechanics~10 (1961), pp.~735-746.

\bibitem{J} W.~B.~Johnson, \emph{A complementary universal conjugate Banach space and its relation to the approximation problem}, Proceedings of the International Symposium on Partial Differential Equations and the Geometry of Normed Linear Spaces (Jerusalem, 1972), Israel J. Math.~13 (1972), pp.~301-310 (1973).

\bibitem{JO} W.~B~Johnson, S.~Ortega Castillo, \emph{The cluster value problem in spaces of continuous functions}, To appear in Proceedings of the AMS (2013).

\bibitem{M} J.~Mujica, \emph{Complex Analysis in Banach Spaces}, North Holland Mathematics Studies, 120, Amsterdam, 1986.


\bibitem{Mc} G.~McDonald, \emph{The maximal ideal space of $H^{\infty}+C$ on the ball in $\mathbb{C}^n$}, Can. Math. J.~31 (1979), pp.~79-86.

\bibitem{S} C.~Stegall, \emph{Banach spaces whose duals contain $\ell_1(G)$ with applications to the study of dual $L_1(\mu)$ spaces}, Trans. Amer. Math. Soc.~176 (1973), pp.~463-477.

\end{thebibliography}
\end{document}